\documentclass[a4paper,11pt,reqno]{amsart}
\usepackage[T2A]{fontenc}
\usepackage[cp1251]{inputenc}
\usepackage[english]{babel}
\usepackage{amsmath,amsthm,amsfonts,amssymb}
\usepackage{enumerate}
\usepackage{graphicx}
\usepackage{color}

\theoremstyle{plain}
\newtheorem{theorem}{Theorem}
\newtheorem{lemma}{Lemma}
\newtheorem{corollary}{Corollary}
\newtheorem*{corollary*}{????????}

\theoremstyle{definition}
\newtheorem{definition}{Definition}

\theoremstyle{remark}
\newtheorem{remark}{Remark}
\newtheorem*{remark*}{Remark}


\begin{document}

\title[Cantor series expansions and packing dimension faithfulness]
{Cantor series expansions \\ and packing dimension faithfulness}
\author[Yu. Kondratiev, M.Lebid, O.Slutskyi, G. Torbin  ]{Yuri Kondratiev$^{1,4}$,
Mykola Lebid$^{2}$,\\ Oleksandr Slutskyi $^{3}$, Grygoriy Torbin$^{4,5}$}
\date{}
\maketitle
\begin{abstract}
The paper is devoted to the development of  general theory of packing measures and dimensions via introducing the notion of <<faithfulness of a packing family for $\dim_P$ calculation>> and the packing analogues of the Billingsley dimension. To this aim we study equivalent definitions of packing dimension and prove theorems which can be considered as  packing analogues of the famous Billingsley's theorems.  The main result of the paper gives necessary and sufficient condition for the packing dimension faithfulness of the family of cylinders generated by the Cantor series expansion. To the best of our knowledge this  is the first known sharp   condition of the packing dimension faithfulness for a class of packing families containing both faithful and non-faithful ones.
\end{abstract}

$^1$Fakult\"{a}t  f\"{u}r Mathematik, Universit\"{a}t Bielefeld, Postfach 10 01 31, D-33501,  Bielefeld (Germany); E-mail:
kondrat@uni-bielefeld.de

 $^2$ Fakult\"{a}t  f\"{u}r Mathematik, Universit\"{a}t Bielefeld, Postfach 10 01 31, D-33501,  Bielefeld (Germany); E-mail: mlebid@math.uni-bielefeld.de

$^{3}$~National Dragomanov Pedagogical University, Pyrogova str. 9, 01030 Kyiv(Ukraine); E-mail: slualexvas@gmail.com

$^{4}$~ National Dragomanov Pedagogical University, Pyrogova str. 9, 01030 Kyiv
(Ukraine)~$^{5}$Institute for Mathematics of NASU, Tereshchenkivs'ka str. 3, 01601 Kyiv (Ukraine); E-mail:
torbin@iam.uni-bonn.de (corresponding author)
\medskip

\medskip


\section{Introduction}
The Hausdorff dimension $\dim_H$ \cite{hausdorff_dim_definition} is the most famous fractal dimension. It is well known that the determination of this dimension is a rather non-trivial problem for many sets and measures (see, e.g., \cite{apt2004,apt2008,a_t_q_star, akpt2011,  falconer_fractal_geometry} and references therein).

The packing dimension $\dim_P$ can be considered as an alternative fractal dimension  \cite{tricot_two_definitions, falconer_fractal_geometry}. It has been introduced only in 1980-s but it is widely known and very useful in the study of fractal sets and measures. Let us stress several reasons for the <<popularity>> of packing dimension.
\begin{enumerate}
\item The packing dimension has all <<good>> properties of the fractal dimension, such as the countable stability (see, e.g., \cite{falconer_fractal_geometry}).
\item  <<The introduction of packing measures (remarkably some 60 years after Hausdorff measures) has led to a greater understanding
of the geometric measure theory of fractals, with packing measures behaving in a way that is ‘dual’ to Hausdorff measures in many respects>> (\cite{falconer_fractal_geometry}, P. 53).


\item Information about $\dim_H$ and $\dim_P$ reflects a level of <<regularity>> resp. <<irregularity>> of a set. Inequality $$\dim_H E \leq \dim_P E$$ is widely known (see, e.g., \cite{falconer_fractal_geometry}). If the inequality above becomes the equality, then the set E is said to be <<regular by Tricot>> \cite{tricot_packing_regularity} and it has many interesting properties (for example, $\dim_H(E\times F)=\dim_H E+\dim_H F$, where $E\times F$ is a Cartesian product of $E$ and $F$).
\end{enumerate}

So, the study of $\dim_P$ together with the $\dim_H$ allows us to know more about the geometric nature and regularity of sets and measures. That is why in many works (see, e.g.,
\cite{
anckar-dim-bounds-of-inv-mes-of-ifs,
attia-barral-hausd-and-pack-spectra-of-rand-walks,
fassler-orponen-restricted-families-of-projections,
holland-zhang-dim-results-for-inhomogenious-moran-sets,
jordan-rams-increasing-digit-subsystem,
joyce_equality_dimh_and_dimp}
and others)
the both dimensions are calculated for considered sets and measures.

There are many approaches to the Hausdorff dimension calculation. One of them is related to the notion of  <<faithfulness of a family of coverings for $\dim_H$ calculation>> (see, e.g., \cite{ailt} and references therein). Roughly speaking, a family $\Phi$ of subsets of the unit interval is faithful for $\dim_H$ calculation on the unit interval  if for any $E \subset [0,1]$ for the correct determination of $\dim_H E$ it is enough to consider coverings of $E$ by sets from $\Phi$. This approach makes the Hausdorff dimension calculation simpler in many cases. It is clear that any comparable net (\cite{R}) generates faithful family of coverings, but there exist faithful nets generating fractional measures which are essentially non-comparable w.r.t. classical Hausdorff measures (\cite{ailt}).

The aim of this paper is to develop general theory of packing measures and dimensions via introducing the notion of <<faithfulness of a packing family for $\dim_P$ calculation>> and the packing analogues of the Billingsley dimension. To this aim we study equivalent definitions of packing dimension and prove theorems which can be considered as  packing analogues of the famous Billingsley's theorems (\cite{B}).  The main result of the paper gives necessary and sufficient condition for the packing dimension faithfulness of the family of cylinders generated by the Cantor series expansion. To the best of our knowledge this  is the first known sharp   condition of the packing dimension faithfulness for a class of packing families containing both faithful and non-faithful ones.


\section{Basic definitions}
Let us shortly recall  main notions related to the Hausdorff and packing dimensions.

\subsection{Faithfulness w.r.t. $\dim_H$ calculation}

Let $\Phi$ be a fine family of coverigs on $[0;1]$, i.e., a family of subsets of $[0;1]$ such that for any $\varepsilon>0$ there exists an at most countable $\varepsilon$-covering $\{E_j\}$ of $[0;1]$ with $E_j\in\Phi$.

\begin{definition}
The $\alpha$-dimensional Hausdorff measure of a set $E\subset [0;1]$ w.r.t. a given fine family of coverings $\Phi$ is defined by
$$
H^\alpha(E,\Phi)=\lim_{\varepsilon\to0}
\left(
\inf
\left\{
\sum_j |E_j|^\alpha
\right\}
\right)=
\lim_{\varepsilon\to0} H^\alpha_\varepsilon(E,\Phi),
$$
where the infimum is taken over all at most countable $\varepsilon$-coverings $\{E_j\}$ of $E$, $E_j\in\Phi$.
\end{definition}

\begin{definition}
The Hausdorff dimension of a set E  w.r.t. $\Phi$ is define by
$$
\dim_H(E,\Phi):=\inf\{\alpha: H^\alpha(E,\Phi)=0\}.
$$
\end{definition}

\begin{definition}
A fine coverings family $\Phi$ is said to be faithful for the Hausdorff dimension calculation on $[0;1]$ if
$$
\dim_H(E,\Phi)=\dim_H(E), \forall E\subset [0;1].
$$
\end{definition}

A historical review of the notion of <<faithfulness for the Hausdorff dimension calculation>>  can be found in \cite{ailt}. In \cite{ailt} authors also obtained general necessary and sufficient conditions for the Hausdorff dimension faithfulness of Vitaly coverings and sharp conditions for the $\dim_H$ faithfulness  for the family of cylinders generated by Cantor series expansions.

\subsection{Faithfulness w.r.t. $\dim_P$  calculation}

The packing dimension $\dim_P$ was introduced by C. Tricot \cite{tricot_two_definitions} at the beginning of 1980 in the following way.

Let $E$ be a subset of a metric space $(M, \rho)$, let $|E|$ be  the diameter of a bounded set $E$.

\begin{definition}
 Let $E\subset M$, $\varepsilon>0$. A finite or countable family $\{E_j\}$ of balls is called  an \textit{$\varepsilon$-packing} of a set $E$ if
\begin{enumerate}
\item $|E_i|\leqslant \varepsilon,\forall i$;
\item A center of any $E_i$ belongs to $E$;
\item $E_i \cap E_j = \varnothing, \forall~  i \neq j$.
\end{enumerate}
\end{definition}

\begin{remark}
The empty set of balls is also a packing of any set.
\end{remark}

\begin{definition}
\label{p-premeasure}
Let $E\subset M$, $\alpha\geqslant0$, $\varepsilon>0$. The \textit{$\alpha$-dimensional packing pre-measure} of bounded set $E$ is defined by
$$
\mathcal{P}^\alpha_{\varepsilon}(E):=\sup\left\{
\sum_i |E_i|^\alpha
\right\},
$$

\noindent where the supremum is taken over all $\varepsilon$-packings $\left\{ E_j \right\}$ of $E$ (if $\{E_j\} = \varnothing,$ then $\mathcal{P}^\alpha_{\varepsilon}(E):=0$).
\end{definition}

\begin{definition}
\label{p-quasimeasure}
\textit{The $\alpha$-dimensional packing quasi-measure} of a set $E$ is defined by
$$
\mathcal{P}^\alpha_{0}(E):=\lim_{\varepsilon\rightarrow 0} \mathcal{P}^\alpha_{\varepsilon}(E).
$$
\end{definition}
Unfortunately, the $\alpha$-dimensional packing quasi-measure is not a measure (to show this it is enough to consider any countable everywhere dense set).

\begin{definition}
\label{p-dimension}
\textit{The $\alpha$-dimensional packing measure} is defined by
$$
\mathcal{P}^\alpha(E):=\inf\left\{
\sum_j \mathcal{P}^\alpha_{0}(E_j): E\subset \bigcup E_j
\right\},
$$

\noindent where the infimum is taken over all at most countable
 coverings $\left\{ E_j \right\}$ of $E$, $E_j\subset \mathbf{M}$.

\end{definition}

\begin{definition}
The nonnegative number
$$
\dim_{P}(E):
=\inf\{\alpha: \mathcal{P}^\alpha(E)=0\}.
$$
\end{definition}

is called the \textit{packing dimension} of a set $E \subset W$.

To simplify the calculation of the packing dimension it is natural to introduce the notion of packing dimension  faithfulness for a countable family $\Phi$ of packings. Proving the packing analogue of the Billingsley's theorems is an additional  motivation to introduce the notion of $\dim_P(E,\Phi)$.

Unfortunately, a direct analogy with $\dim_H(E,\Phi)$ does not lead to applicable results. To explain this remark let us consider any countable family $\Phi$ of balls, and let
$$
\mathcal{P}^\alpha_\varepsilon(E,\Phi)=
\sup\left\{
\sum_i |E_i|^\alpha
\right\},
$$
where supremum is taken over all possible packings $\{E_i\}$ of  a set $E$ with $E_i \in \Phi$. Then we  define quasi-measure, measure and dimension by definitions \ref{p-premeasure}, \ref{p-quasimeasure} and \ref{p-dimension} respectively. In such a case every  family $\Phi$ is not faithful for the  packing dimension calculation on $[0;1]$. To prove this we consider the set  $C_\Phi$ of  centers of all balls from $\Phi$, and then define $E_0:=[0;1] \setminus C_\Phi$. It is clear that $\dim_P(E_0)=1$. On the other hand $\dim_P(E_0,\Phi)=0,$ because there are no packings of the set $E_0$ by balls from $\Phi$.

 Therefore, families of cylinders generated by  $s$-adic, $Q$, $Q^*$, $\tilde{Q}$-expansions can not be faithful for the classical packing dimension calculation.

 It is clear that the condition <<centers of balls is in the figure, dimension of which is calculating>> is  the main reason of this problem with     classical (centered) packing definition. That is why we introduce a new notion of <<uncentered packing>> and respectively <<uncentered packing dimension>> $\dim_{P(unc)}$. The $\dim_{P(unc)}$ definition is similar to $\dim_P$ definition, except that condition <<centers of balls in the packing is in the set $E$>> is replaced by <<every ball from the packing has a non-empty intersection with $E$>>.

Next we prove that
$$
\dim_{P(unc)} E=\dim_P E, \forall E
$$
in a wide class of metric spaces including $\mathbb{R}^n$.

Based on the notion of $\dim_{P(unc)}$, we introduce notions of $\dim_P(E,\Phi)$ and $\dim_P(E,\Phi,\mu)$, and the notion <<faithfulness of a  family of balls for packing (generally speaking, uncentered, but in $\mathbb{R}^n$ we drop this word) dimension calculation>>.

\section{ Equivalent definitions and generalizations of packing dimension.}

\subsection{Uncentered packing dimension}
 The notions of <<Hausdorff dimension with respect to the family of covering>> and <<Billingsly dimension>> are well known generalizations of the classical Hausdorff dimension. They give a powerful tool for the determination and estimations of the Hausdorff dimension of sets and probability measures. As we explained above, the condition <<the centers of packing balls belong to set>> does not give a possibility to develop similar tools for the packing case. Because of this reason we  introduce a notion of uncentered packing dimension.

\begin{definition}
 Let $E\subset M$, $\varepsilon>0$. A finite or a countable family $\{E_j\}$ of open balls is called  an \textit{uncentered $\varepsilon$-packing} of a set $E$ if

\begin{enumerate}
\item $|E_i|\leqslant \varepsilon,\forall i$;
\item $E_i \cap E\neq\varnothing$;
\item $E_i \cap E_j = \varnothing~ \forall~i \neq j$.
\end{enumerate}
\end{definition}

\begin{remark}
The empty set of balls is also an uncentered packing of any set.
\end{remark}

\begin{definition}
Let $E\subset M$, $\alpha\geqslant0$, $\varepsilon>0$. The \textit{uncentered $\alpha$-dimensional packing pre-measure} of a bounded set $E$  is defined by
$$
\mathcal{P}^\alpha_{\varepsilon(unc)}(E):=\sup\left\{
\sum_i |E_i|^\alpha
\right\},
$$

\noindent where the supremum is taken over all at most countable
 uncentered $\varepsilon$-packings $\left\{ E_j \right\}$ of $E$ (if $\{E_j\} = \varnothing$, then $\mathcal{P}^\alpha_{\varepsilon (unc)}(E):=0$).

\end{definition}

Directly from the definition it follows that
\begin{enumerate}
\item \textbf{Monotonicity.} If $E_1 \subset E_2$, then $\mathcal{P}^\alpha_{\varepsilon(unc)}(E_1) \leqslant \mathcal{P}^\alpha_{\varepsilon(unc)}(E_2)$;
\item \textbf{Sub-additivity.} $\mathcal{P}^\alpha_{\varepsilon(unc)}(E_1 \cup E_2) \leqslant \mathcal{P}^\alpha_{\varepsilon(unc)}(E_1)+\mathcal{P}^\alpha_{\varepsilon(unc)}(E_2)$;
\item $\forall \delta>0:\mathcal{P}^{\alpha+\delta}_{\varepsilon(unc)}(E) \leqslant \mathcal{P}^\alpha_{\varepsilon(unc)}(E)\cdot \varepsilon^\delta$.
\end{enumerate}

\begin{definition}
\textit{The uncentered $\alpha$-dimensional packing quasi-measure} of a set $E$ is defined by

$$
\mathcal{P}^\alpha_{0(unc)}(E):=\lim_{\varepsilon\rightarrow0} \mathcal{P}^\alpha_{\varepsilon(unc)}(E).
$$
\end{definition}

Let us formulate basic properties of the uncentered $\alpha$-dimensional packing quasi-measure.
\begin{enumerate}
\item \textbf{Monotonicity.} If $E_1 \subset E_2$, then $\mathcal{P}^\alpha_{0(unc)}(E_1) \leqslant \mathcal{P}^\alpha_{0(unc)}(E_2)$;
\item \textbf{Sub-additivity.}$ \mathcal{P}^\alpha_{0(unc)}(E_1 \cup E_2) \leqslant \mathcal{P}^\alpha_{0(unc)}(E_1)+\mathcal{P}^\alpha_{0(unc)}(E_2)$;
\item \textbf{The set function $\mathcal{P}^\alpha_{0(unc)}(E)$ is not $\sigma$-additive.} There is a family of sets $E_1, E_2, \dots, E_k, \dots$ such that $\mathcal{P}^\alpha_{0(unc)}(\bigcup\limits_{i=1}^{\infty} E_i)>\sum\limits_{i=0}^{\infty} \mathcal{P}^\alpha_{0(unc)}(E_i)$;
\item If $\mathcal{P}^\alpha_{0(unc)}(E)<\infty$, then  $\mathcal{P}^{\alpha+\delta}_{0(unc)}(E)=0, ~~~ \forall \delta>0$;
\item If $\mathcal{P}^\alpha_{0(unc)}(E)>0$ for some positive $\alpha$, then  $\mathcal{P}^{\alpha-\delta}_{0(unc)}(E)=+\infty, ~~~ \forall \delta\in(0;\alpha). $
\end{enumerate}

\begin{definition}
\textit{The uncentered $\alpha$-dimensional packing measure} is defined by
$$
\mathcal{P}^\alpha_{(unc)}(E):=\inf\left\{
\sum_j \mathcal{P}^\alpha_{0(unc)}(E_j): E\subset \bigcup E_j
\right\},
$$

\noindent where the infimum is taken over all at most  countable uncentered
 coverings $\left\{ E_j \right\}$ of $E$, $E_j\subset \mathbf{M}$.

\end{definition}


Let us formulate basic properties of the uncentered $\alpha$-dimensional packing measure.
\begin{enumerate}
\item \textbf{Monotonicity.} If $E_1 \subset E_2$, then $\mathcal{P}^\alpha_{(unc)}(E_1) \leqslant \mathcal{P}^\alpha_{(unc)}(E_2)$;
\item \textbf{$\sigma$-sub-additivity.}
$$
\mathcal{P}^\alpha_{(unc)}(\bigcup\limits_i E_i) \leqslant \sum\limits_i\mathcal{P}^\alpha_{(unc)}(E_i),~E_i \subset M,~ \forall i \in \mathbb{N}.
$$
\item If $\mathcal{P}^\alpha_{(unc)}(E)<\infty$, then  $\mathcal{P}^{\alpha+\delta}_{(unc)}(E)=0,  ~~~ \forall \delta>0$;

\item If $\mathcal{P}^\alpha_{(unc)}(E)>0$ for some  $\alpha>0$, then  $\mathcal{P}^{\alpha-\delta}_{(unc)}(E)=+\infty, ~~~ \forall \delta\in(0;\alpha)$.
\end{enumerate}


\begin{definition}
The nonnegative number
$$
\dim_{P(unc)}(E):
=\inf\{\alpha: \mathcal{P}^\alpha_{(unc)}(E)=0\}.
$$

is called the \textit{uncentered packing dimension} of a set $E \subset M$.

\end{definition}

By using standard approach one can easily prove basic properties of the uncentered packing dimension.
\begin{enumerate}
\item \textbf{Monotonicity.} If $E_1 \subset E_2$, then $\dim_{P(unc)}(E_1) \leqslant \dim_{P(unc)}(E_2)$;
\item \textbf{Countable stability.}
$$
\dim_{P(unc)}(\bigcup\limits_i E_i)=\sup\limits_i \dim_{P(unc)}(E_i),~ E_i \subset M,~ \forall i \in \mathbb{N} .
$$
\end{enumerate}

\begin{theorem}
Let $(M,\rho)$ be a metric space. Suppose that there exists a positive integer   $C$ such that any open ball $I$ contains at most $C$ non-intersecting open balls whose diameters are equal to $\frac{1}{8} |I|$. Then
$$
\dim_{P(unc)}(E)=\dim_P(E), ~~~ \forall E \subset M.
$$
\end{theorem}
\begin{proof}

\textbf{Step 1.} Firstly let us prove that $\dim_{P(unc)}(E)\geqslant \dim_P(E)$.
From  definitions of pre-measures  and  properties of suprema it follows that
$$
\mathcal{P}^\alpha_{r(unc)}(A)\geqslant\mathcal{P}^\alpha_r(A), ~ \forall A \subset M, ~\forall \alpha>0, ~\forall r>0.
$$
Taking the limit, we get
$$
\mathcal{P}^\alpha_{0(unc)}(A)\geqslant\mathcal{P}^\alpha_0(A), ~ \forall A \subset M, ~\forall \alpha>0.
$$

So, for any  $ E \subset M$ and for any covering $\{E_j\}$ of $E$ we get

$$
\mathcal{P}^\alpha_{0(unc)}(E_j)\geqslant\mathcal{P}^\alpha_0(E_j), ~\forall \alpha>0.
$$

Hence
$$
\mathcal{P}^\alpha_{(unc)}(E)\geqslant\mathcal{P}^\alpha(E), ~ \forall E \subset M,  ~\forall \alpha>0.
$$

Let $\dim_{P(unc)}(E)=\alpha_0$.  By the definition of $\dim_{P(unc)}(E)$, we have
$$
\mathcal{P}^{\alpha_0+\varepsilon}_{(unc)}(E)=0,~~~ \forall\varepsilon>0.
$$
Therefore
$$
\mathcal{P}^{\alpha_0+\varepsilon}_0(E)=0,~~~ \forall\varepsilon>0,
$$
and consequently
$$
\dim_P(E)\leqslant \alpha_0.
$$
So, $\dim_{P(unc)}(E)\geqslant \dim_P(E)$.

\textbf{Step 2.} Let us show that  $\dim_{P(unc)}(E)\leqslant \dim_P(E)$.

If $\dim_{P(unc)}(E)=0$, then the statement  is obvious.

Let us work with the  case where $\dim_{P(unc)}(E)\neq 0$. Let us choose positive reals $t$ and $s$ such that $0<t<s<\dim_{P(unc)}(E)$.

Since
$
s<\dim_{P(unc)}(E),
$
 we have
$$
\mathcal{P}^s_{(unc)}(E)=+\infty,
$$
and, therefore,
$$\mathcal{P}^s_{0(unc)}(E)=+\infty.
$$

Hence,
$$
\forall r>0: \mathcal{P}^s_{r(unc)}(E)=+\infty.
$$
So, there exists an uncentered packing $V:=\{E_i\}$ of the set $E$, with
\begin{equation}
\label{eq:dimp_unc_equal_to_dimp_first}
\sum_i |E_i|^s>1.
\end{equation}
For any $r \in (0,1)$ let us split the packing $V$ into disjoint classes:
$$
V_k:=\left\{
E_i: 2^{-k-1}\leqslant |E_i| < 2^{-k}
\right\}, ~ k \in \{0,1,2,...\}.
$$
Let $n_k$ be the number of balls in  $V_k$. Let us prove that
$$
\exists k_0: n_{k_0} \geqslant 2^{k_0 t}(1-2^{t-s}).
$$

To obtain a contradiction, suppose that
$$
n_k < 2^{kt}(1-2^{t-s}), \forall k.
$$
Then
$$
\sum_i |E_i|^s<\sum_{k=0}^{\infty} 2^{-ks}\cdot n_k<\sum_{k=0}^{\infty} 2^{-ks} \cdot 2^{kt} (1-2^{t-s})=(1-2^{t-s}) \cdot \sum_{k=0}^{\infty} (2^{t-s})^k=1,
$$
which contradicts our assumption \eqref{eq:dimp_unc_equal_to_dimp_first}. So, there exists $k_0$ with $n_{k_0} \geqslant 2^{k_0 t}(1-2^{t-s}).$

Let us work with $V_{k_0}$. We denote by $A_1, A_2, \dots, A_{n_{k_0}}$ the balls in $V_{k_0}$, i. e.,
$$
V_{k_0}=\left\{
A_1, A_2, \dots, A_{n_{k_0}}
\right\}.
$$

Fix $r:=2^{-k_0-1}$. Then a radius of any $A_i$ is less then $r$.
Let $T_i$ be a point of $A_i$ such that  $T_i \in A_i \cap E$. Let $V'$ be the set of balls with the centers $T_i$, and the radius $r$, i.e.,
$$
V'=\{A'_i: A'_i=B(T_i,r)\}.
$$
Fix
$$
V^*=\{A^*_i: A^*_i=B(T_i,4r)\}.
$$


Let us split the set $V'$ into families $K_1, K_2, \dots, K_l$ by using the following procedure.

\begin{enumerate}
\item Let $A'_{j_1}:=A'_1$ and let the family $K_1$ consists of $A'_{j_1}$ and all other balls $A'_{i} \in V'$ such that $A'_{i} \cap A'_{j_1} \neq \varnothing$.

\item We  choose an arbitrary ball $A'_{j_2} \in V'\setminus K_1$ and define  $K_2$ to be the family consisting of $A'_{j_2}$ and  all other balls $A'_{i} \in V'\setminus K_1$ such that $A'_{i} \cap  A'_{j_2} \neq \varnothing$.

\item And so on.

\item We will continue this process until $V'\setminus (K_1 \cup K_2 \cup ... \cup K_l) = \varnothing$. Since the number of elements in $V'$ is finite, the above mentioned number $l$ is not greater then $n_{k_0}$. \end{enumerate}

It is clear that if  $A'_i \bigcap A'_j \neq \emptyset$ (i.e., $\rho(T_i,T_j)\leqslant 2r$), then $A_j \subset A^*_i$.

A radius of $A_j$ is bigger then $r/2$. From the assumprion of the theorem it follows that there are not more then $C$ disjoint balls with radius $r/2$ in a ball with radius $4r$.   Therefore there are not more then $C$ balls in any family $K_i$. So, $\frac{n_{k_0}}{C}\leq l$.

From the construction of families $\{K_i\}$ it follows that if  $i<m$, then  balls $A'_{j_i}$ and  $A'_{j_m}$ do not intersect each other.

Therefore,
$$
V''=\left\{
A'_{j_1}, A'_{j_2}, \dots, A'_{j_l}
\right\}
$$
is a centered packing of a set $E$ and
$$
\sum_{i=1}^l |A'_{j_i}|^t
= l\cdot (2r)^t \geq  n_{k_0}\cdot\frac{2^{-k_0 t}}{C} \geqslant 2^{k_0 t}(1-2^{t-s})\cdot \frac{2^{-k_0 t}}{C} = \frac{1-2^{t-s}}{C}.
$$
Hence, $$
\mathcal{P}^t_{2^{-k_0}}(E)\geqslant \frac{1-2^{t-s}}{C}.
$$

By the inequality  $2^{-k_0}<r$, we get
$$
\mathcal{P}^t_{r}(E)\geqslant \frac{1-2^{t-s}}{C}, \forall r>0,
$$
and, therefore,
$$
\mathcal{P}^t_{0}(E)\geqslant \frac{1-2^{t-s}}{C}.
$$

Let us show that $\mathcal{P}^t(E)\geqslant\frac{1-2^{t-s}}{C}$. Let us recall that

$$
\mathcal{P}^t(E)=\inf\left\{
\sum_j \mathcal{P}^t_{0}(E_j): E\subset \bigcup E_j
\right\},
$$
where the infinitum is taken over all at most countable coverings $E_j$ of set $E$.

Let $\{E_j\}$ be an arbitrary at most countable coverings of $E$. Since $\dim_{P(unc)}(E)>s$, from the countable stability of uncentered packing dimension $\dim_{P(unc)}$ it follows that there exists a  $j_0$ such that  $\dim_{P(unc)}(E_{j_0})>s$. It is clear that in such a case we have
$$
\mathcal{P}^s_{(unc)}(E_{j_0})=+\infty,
$$
and, therefore,
$$
\mathcal{P}^s_{0(unc)}(E_{j_0})=+\infty
$$

Repeating the same arguments as we have already done in this proof for the set $E$,
we get
$$
\mathcal{P}^t_{0}(E_{j_0})\geqslant \frac{1-2^{t-s}}{C}.
$$
Therefore,
$$
\sum_j \mathcal{P}^t_{0}(E_j)\geqslant \frac{1-2^{t-s}}{C}.
$$

Since the latter inequality is true for an arbitrary covering $\{E_j\}$ of a set $E$, we conclude that
$$
\mathcal{P}^t(E)\geqslant \frac{1-2^{t-s}}{C}.
$$
So,
$$
\dim_P(E)\geqslant t.
$$
Since the real numbers  $t$ and $s$ can be chosen arbitrarily close to $\dim_{P(unc)}(E)$, we  get the desired inequality $\dim_P(E) \geqslant\dim_{P(unc)}(E)$, which completes the proof.
\end{proof}
\begin{corollary}
If $M=\mathbb{R}^n$, then $\dim_{P(unc)}(E)=\dim_P(E)$.
\end{corollary}

\bigskip
\bigskip
\bigskip
\bigskip
\bigskip

\subsection{Packing dimension with respect to a family of sets}
Let $\Phi$ be a family of balls in a metric space $(M, \rho)$.

\begin{definition}
Let $E\subset M$, $\alpha\geqslant0$, $\varepsilon>0$. Then \textit{$\alpha$-dimensional packing pre-measure} of a bounded set $E$ with respect to  $\Phi$ is defined by

$$
\mathcal{P}^\alpha_{\varepsilon}(E,\Phi):=\sup\left\{
\sum_i |E_i|^\alpha
\right\},
$$

\noindent where the supremum is taken over all uncentered
 $\varepsilon$-packings $\left\{ E_i \right\} \subset \Phi$ of $E$ (if $\{E_i\} = \varnothing,$ then  $\mathcal{P}^\alpha_{\varepsilon}(E,\Phi):=0$).
\end{definition}

The following properties of the $\alpha$-dimensional packing pre-measure w. r. t. family $\Phi$ follow directly from the definition.
\begin{enumerate}
\item \textbf{Monotonicity.} If $E_1 \subset E_2$, then $\mathcal{P}^\alpha_{\varepsilon}(E_1,\Phi) \leqslant \mathcal{P}^\alpha_{\varepsilon}(E_2,\Phi)$;
\item \textbf{Sub-additivity.} $\mathcal{P}^\alpha_{\varepsilon}(E_1 \cup E_2,\Phi) \leqslant \mathcal{P}^\alpha_{\varepsilon}(E_1,\Phi)+\mathcal{P}^\alpha_{\varepsilon}(E_2,\Phi)$;
\item $\forall \delta>0:\mathcal{P}^{\alpha+\delta}_\varepsilon(E,\Phi) \leqslant \mathcal{P}^\alpha_\varepsilon(E,\Phi)\cdot \varepsilon^\delta$.
\end{enumerate}

\begin{definition}
\textit{The $\alpha$-dimensional packing quasi-measure} of a set $E$
 w. r. t. $\Phi$ is defined by
$$
\mathcal{P}^\alpha_{0}(E,\Phi):=\lim_{\varepsilon\rightarrow0} \mathcal{P}^\alpha_{\varepsilon}(E,\Phi).
$$
\end{definition}

Let us formulate basic properties of the $\alpha$-dimensional packing quasi-measure w. r. t. $\Phi$.
\begin{enumerate}
\item \textbf{Monotonicity.} If $E_1 \subset E_2$, then $\mathcal{P}^\alpha_{0}(E_1,\Phi) \leqslant \mathcal{P}^\alpha_{0}(E_2,\Phi)$;
\item \textbf{Sub-additivity.}$  \mathcal{P}^\alpha_{0}(E_1 \cup E_2,\Phi) \leqslant \mathcal{P}^\alpha_{0}(E_1,\Phi)+\mathcal{P}^\alpha_{0}(E_2,\Phi)$;
\item If $\mathcal{P}^\alpha_0(E,\Phi)<\infty$, then  $\mathcal{P}^{\alpha+\delta}_0(E,\Phi)=0, ~~$ $\forall \delta>0$;
\item If $\mathcal{P}^\alpha_0(E)>0$ and $\alpha>0$, then  $\mathcal{P}^{\alpha-\delta}_0(E,\Phi)=+\infty,~~~$ $\forall \delta\in(0;\alpha)$.
\end{enumerate}

\begin{definition}
\textit{The $\alpha$-dimensional packing measure} w. r. t. $\Phi$ is defined by
$$
\mathcal{P}^\alpha(E,\Phi):=\inf\left\{
\sum_j \mathcal{P}^\alpha_{0}(E_j,\Phi): E\subset \bigcup E_j
\right\},
$$

\noindent where the infimum is taken over all at most countable
 coverings $\left\{ E_j \right\}$ of $E$, $E_j\subset \mathbf{M}$.
\end{definition}

Let us formulate basic properties of the $\alpha$-dimensional packing measure w.r.t. $\Phi$.

\begin{enumerate}
\item \textbf{Monotonicity.}  If $E_1 \subset E_2$, then $\mathcal{P}^\alpha(E_1,\Phi) \leqslant \mathcal{P}^\alpha(E_2,\Phi)$;
\item \textbf{$\sigma$-Sub-additivity.}
$$
\mathcal{P}^\alpha(\cup_i E_i,\Phi) \leqslant \sum_i\mathcal{P}^\alpha(E_i,\Phi), ~E_i \subset W,~ \forall i \in \mathbb{N}..
$$
\item If $\mathcal{P}^\alpha(E,\Phi)<\infty$, then  $\mathcal{P}^{\alpha+\delta}(E,\Phi)=0, ~~~$  $\forall \delta>0$;
\item If $\mathcal{P}^\alpha(E,\Phi)>0$ and $\alpha>0$, then  $\mathcal{P}^{\alpha-\delta}(E,\Phi)=+\infty, ~~~$ $\forall \delta\in(0;\alpha)$.
\end{enumerate}

\begin{definition}
The nonnegative number
$$
\dim_{P}(E,\Phi):
=\inf\{\alpha: \mathcal{P}^\alpha(E,\Phi)=0\}
$$
is called the \textit{packing dimension} of a set $E \subset M$ w. r. t. $\Phi$.
\end{definition}

By using standard approach one can easily prove monotonicity and countable stability  of the packing dimension w.r.t. $\Phi$, i.e.,
$$
\dim_P(\bigcup\limits_i E_i,\Phi)=\sup\limits_i \dim_P(E_i,\Phi), ~E_i \subset M,~ \forall i \in \mathbb{N} .
$$

\begin{lemma}
\label{dimpfi_leq_dimp}
$$\dim_{P}(E,\Phi)\leqslant\dim_{P(unc)}(E).$$
\end{lemma}
\begin{proof}
Let $\Phi_0$ be a family of all open balls of $M$. Then $$\mathcal{P}^\alpha_{r(unc)}(E)=\mathcal{P}^\alpha_{r}(E,\Phi_0).$$

\noindent Since $\Phi\subseteq\Phi_0$, we see that
$$\mathcal{P}^\alpha_{r}(E,\Phi)\leqslant\mathcal{P}^\alpha_{r}(E,\Phi_0).$$

\noindent By the inequality for packing pre-measures, it follows that
$$\dim_{P}(E,\Phi)\leqslant\dim_{P(unc)}(E).$$
\end{proof}

\subsection{Packing dimension w. r. t. a family of sets and a measure}

Let $\Phi$ be a family of open balls in a metric space $(M, \rho)$ and let $\mu$ be  a continuous measure.

\begin{definition}
Let $E\subset M$, $\alpha\geqslant0$, $\varepsilon>0$. Then \textit{$\alpha$-dimensional packing pre-measure} of a bounded set $E$ with respect to  $\Phi$ and   $\mu$ is defined by
$$
\mathcal{P}^\alpha_{\varepsilon}(E,\Phi,\mu):=\sup\left\{
\sum_i (\mu(E_i))^\alpha\right\},
$$
\noindent where the supremum is taken over all at most countable (uncentered) $\varepsilon$-packings $\left\{ E_i \right\} \subset \Phi$ of $E$ (if $\{E_j\} = \varnothing$, then  $\mathcal{P}^\alpha_{\varepsilon}(E,\Phi,\mu):=0$).
\end{definition}

The following properties of the $\alpha$-dimensional packing pre-measure w. r. t. family $\Phi$  and measure $\mu$ follows directly from the definition.

\begin{enumerate}
\item \textbf{Monotonicity.} If $E_1 \subset E_2$, then $\mathcal{P}^\alpha_{\varepsilon}(E_1,\Phi,\mu) \leqslant \mathcal{P}^\alpha_{\varepsilon}(E_2,\Phi,\mu)$;
\item \textbf{Sub-additivity.}  $\mathcal{P}^\alpha_{\varepsilon}(E_1 \cup E_2,\Phi,\mu) \leqslant \mathcal{P}^\alpha_{\varepsilon}(E_1,\Phi,\mu)+\mathcal{P}^\alpha_{\varepsilon}(E_2,\Phi,\mu)$;
\item If for a given $\varepsilon>0$ there exists a $c(\varepsilon)>0$ such that $\mu(I)<c(\varepsilon)$ for any ball $I$ with $|I|<\varepsilon$, then
$$\mathcal{P}^{\alpha+\delta}_\varepsilon(E,\Phi,\mu) \leqslant \mathcal{P}^\alpha_\varepsilon(E,\Phi,\mu)\cdot (c(\varepsilon))^\delta, ~~~ \forall \delta>0.$$

\end{enumerate}

\begin{definition}
\textit{The $\alpha$-dimensional packing quasi-measure} of a set $E$
 w. r. t. $\Phi$ and $\mu$ is defined by
$$
\mathcal{P}^\alpha_{0}(E,\Phi,\mu):=\lim_{\varepsilon\rightarrow0} \mathcal{P}^\alpha_{\varepsilon}(E,\Phi,\mu).
$$
\end{definition}

Let us formulate basic properties of the $\alpha$-dimensional packing quasi-measure w. r. t. $\Phi$ and $\mu$.
\begin{enumerate}
\item \textbf{Monotonicity.} If $E_1 \subset E_2$, then $\mathcal{P}^\alpha_{0}(E_1,\Phi,\mu) \leqslant \mathcal{P}^\alpha_{0}(E_2,\Phi,\mu)$;

\item \textbf{Sub-additivity.}  $\mathcal{P}^\alpha_{0}(E_1 \cup E_2,\Phi,\mu) \leqslant \mathcal{P}^\alpha_{0}(E_1,\Phi,\mu)+\mathcal{P}^\alpha_{0}(E_2,\Phi,\mu)$;

\item If $\mathcal{P}^\alpha_0(E,\Phi,\mu)<\infty$ and $\lim\limits_{\varepsilon\to0} c(\varepsilon)=0$ (the function $c(\varepsilon)$ has been defined above)), then  $\mathcal{P}^{\alpha+\delta}_0(E,\Phi,\mu)=0, ~~~$ $\forall \delta>0$;
\item If $\mathcal{P}^\alpha_0(E)>0$ and $\lim\limits_{\varepsilon\to0} c(\varepsilon)=0$ and $\alpha>0$, then  $\mathcal{P}^{\alpha-\delta}_0(E,\Phi,\mu)=+\infty, ~~~$ $\forall \delta\in(0;\alpha)$.
\end{enumerate}

\begin{definition}
\textit{The $\alpha$-dimensional packing measure} w. r. t.  $\Phi$  and $\mu$ is defined by
$$
\mathcal{P}^\alpha(E,\Phi,\mu):=\inf\left\{
\sum_j \mathcal{P}^\alpha_{0}(E_j,\Phi,\mu): E\subset \bigcup E_j
\right\},
$$
\noindent where the infimum is taken over all at most countable
 coverings $\left\{ E_j \right\}$ of $E$, $E_j\subset \mathbf{M}$.
\end{definition}

\begin{enumerate}
\item \textbf{Monotonicity.}  If $E_1 \subset E_2$, then $\mathcal{P}^\alpha(E_1,\Phi,\mu) \leqslant \mathcal{P}^\alpha(E_2,\Phi,\mu)$;
\item \textbf{$\sigma$-Sub-additivity.}
$$
\mathcal{P}^\alpha(\cup_i E_i,\Phi,\mu) \leqslant \sum_i\mathcal{P}^\alpha(E_i,\Phi,\mu), ~E_i \subset W,~ \forall i \in \mathbb{N}.
$$
\item If $\lim\limits_{\varepsilon\to0} c(\varepsilon)=0$ and $\mathcal{P}^\alpha(E,\Phi,\mu)<\infty$, then  $\mathcal{P}^{\alpha+\delta}(E,\Phi,\mu)=0, ~~$ $\forall \delta>0$;

\item If $\mathcal{P}^\alpha(E,\Phi)<\infty$, then  $\mathcal{P}^{\alpha+\delta}(E,\Phi)=0,~~~$ $\forall \delta>0$;
\item If $\mathcal{P}^\alpha(E,\Phi)>0$ and $\alpha>0$, then  $\mathcal{P}^{\alpha-\delta}(E,\Phi)=+\infty, ~~~$ $\forall \delta\in(0;\alpha)$.
\end{enumerate}

\begin{definition}
The nonnegative number
$$
\dim_{P}(E,\Phi,\mu):
=\inf\{\alpha: \mathcal{P}^\alpha(E,\Phi,\mu)=0\}.
$$
is called the \textit{packing dimension} of a set $E \subset W$ w. r. t. $\Phi$ and a measure $\mu$.
\end{definition}

By using the same techniques as for the classical packing dimension one can prove monotonicity and countable stability of the packing dimension w.r.t. $\Phi$ and $\mu$, i.e., $$
\dim_P(\cup_i E_i,\Phi,\mu)=\sup_i \dim_P(E_i,\Phi,\mu), ~E_i \subset M,~ \forall i \in \mathbb{N}.
$$

\begin{remark}
If $M\subset \mathbb{R}^1$ and $\mu$ is a Lebesgue measure ($\mu=\lambda$), then $\dim_P(E,\Phi,\mu)=\dim_P(E,\Phi)$.
\end{remark}

\subsection{Analogue of Billingsley's theorem for the packing dimension}

There are many types of expansions of real numbers over finite as well as infinite alphabets (see, e.g., \cite{Schweiger, Galambos, ABPT, AKNT2,  AT, akpt2011, a_t_q_star} and references therein). Each expansion generates the corresponding procedure of partitions and the family of basic cylinders. For a given expansion of real numbers over an alphabet $A$
$$ 
x=\Delta_{\alpha_1(x) \alpha_2(x) \ldots \alpha_n(x) \ldots ...},  ~~~ \alpha_k(x) \in A,
$$ 
let  $\Delta_n(x)$ be the cylinder of n-th rank containing $x$.

\begin{theorem}
\label{billingsley_theorem_for_dimp}
  Let $\mu$ and  $\nu$ be continuous measures on $[0;1]$ and $\Delta_n(x)$ be cylinder of $n$-th rank containing a point $x \in [0;1]$. Fix $\delta>0$. Let
$$
E\subset \left\{
x:\limsup_{n\to\infty} \frac{\ln\mu(\Delta_n(x))}{\ln\nu(\Delta_n(x))}\leqslant \delta
\right\},
$$

then
$$
\dim_P(E,\Phi,\nu)\leqslant \delta\cdot\dim_P(E,\Phi,\mu).
$$
\end{theorem}
\begin{proof}
At first, we shall prove the theorem under more strong assumptions
\begin{equation}
\label{billingsley-theorem-assumption}
\exists n_0: ~~~ \frac{\ln\mu(\Delta_n(x))}{\ln\nu(\Delta_n(x))}\leqslant \delta, \forall x \in E, ~ \forall n>n_0.
\end{equation}
Then for any $n>n_0$ and for any positive $\alpha$ we have
$$
(\mu(\Delta_n(x)))^\alpha \geqslant (\nu(\Delta_n(x)))^{\alpha\delta}.
$$
Therefore
$$
\mathcal{P}^\alpha_\varepsilon(E,\Phi,\mu)
\geqslant
\mathcal{P}^{\alpha\delta}_\varepsilon(E,\Phi,\nu)
$$
for all small enough $\varepsilon$.

Taking the limit as $\varepsilon\to0$, we have
$$
\mathcal{P}^\alpha_0(E,\Phi,\mu)
\geqslant
\mathcal{P}^{\alpha\delta}_0(E,\Phi,\nu).
$$
Consequently,
$$
\mathcal{P}^\alpha(E,\Phi,\mu)
\geqslant
\mathcal{P}^{\alpha\delta}(E,\Phi,\nu).
$$
Let $\alpha_0:=\dim_P(E,\Phi,\mu)$. Let $\alpha>\alpha_0$ be an arbitrary number. Then $\mathcal{P}^\alpha(E,\Phi,\mu)=0$ and $\mathcal{P}^{\alpha\delta}(E,\Phi,\nu)=0$. Hence $\dim_P(E,\Phi,\nu)\leqslant\alpha\delta$. Therefore,
$$
\dim_P(E,\Phi,\nu)\leqslant \delta\cdot\dim_P(E,\Phi,\mu).
$$

Let
$$
E\subset \left\{
x:\limsup_{n\to\infty} \frac{\ln\mu(\Delta_n(x))}{\ln\nu(\Delta_n(x))}\leqslant \delta
\right\}.
$$

Then
$$
\forall x\in E,\forall \varepsilon>0,\exists N(x,\varepsilon): ~ \frac{\ln\mu(\Delta_n(x))}{\ln\nu(\Delta_n(x))}\leqslant \delta+\varepsilon, ~~~ \forall n>N(x,\varepsilon).
$$
Let $N_0(x, \varepsilon)$ be the minimal number with this property.
For a given  $\varepsilon>0$, let
$$
E_m=\left\{
x: N_0(x,\varepsilon) \leqslant m
\right\},
$$
where $m \in \mathbb{N}$.

By the definition, we have
$$
E_m \subset \left\{
x: \frac{\ln\mu(\Delta_n(x))}{\ln\nu(\Delta_n(x))}\leqslant \delta+\varepsilon, \forall n>m
\right\}.
$$

Hence
$$
\dim_P(E_m,\Phi,\nu) \leqslant (\delta+\varepsilon)\cdot\dim_P(E_m,\Phi,\mu),\forall m\in\mathbb{N}.
$$

From the  countable stability it follows that
$$
\dim_P(E,\Phi,\nu) \leqslant (\delta+\varepsilon)\cdot\dim_P(E,\Phi,\mu).
$$
Since $\varepsilon$ can be chosen arbitrarily small, we have
$$
\dim_P(E,\Phi,\nu) \leqslant \delta\cdot\dim_P(E,\Phi,\mu),
$$
which proves the theorem.
\end{proof}

 It is necessary to mention that a simple version of this theorem has been  proven by M. Das \cite{das-billingsley-packing-dimension} in 2008.

\section{On faithfulness of a packing family for the packing dimension calculation}

\subsection{Sharp conditions for faithfulness of packing families generated by Cantor series expansions}

 Let us recall that for a given sequence $\left\{n_{k} \right\}_{k=1}^{\infty } $ with $n_{k} \in \mathbb{N} \backslash \{ 1\} ,\, k \in
\mathbb{N}$ the  expression of $x \in [0,1]$ in the following form
\[x=\sum_{k=1}^{\infty}\frac{\alpha_{k}}{n_{1} \cdot n_{2} \cdot \ldots \cdot n_{k}}=:\Delta_{\alpha_{1}\alpha_{2}...\alpha_{k}...},~\alpha_{k} \in \{0,~1,~...,~n_{k}-1\}\]
is said to be the Cantor series expansion of $x$.  These expansions, which have been initially studied by G. Cantor in 1869 (see., e.g. \cite{Cantor1869}), are natural generalizations of the classical $s$-adic expansion for reals. In \cite{Airey} authors mentioned that <<G. Cantor’s motivation to study the Cantor series expansions was to extend the well known proof of the irrationality of the number $e = \sum\limits_{n=0}^{\infty} \frac{1}{n!}$ to a larger class of numbers. Results along these lines may be found in the monograph of J. Galambos \cite{Galambos}>>. Cantor series expansions have been intensively studied from different points of view during last century (see, e.g., \cite{ManceDiss, Schweiger} and references therein). They can be used to get simple proofs of irrationality of some famous constants (see, e.g., \cite{drobot}). A lot of efforts were spent by many mathematicians to find sharp conditions for rationality resp. irrationality of real numbers in terms of the sequence $\{n_k\}$, but this problem is still open. A series of research papers related to the normality of real numbers in terms of Cantor series expansions and fractal properties of subsets of non-normal numbers have been published during last decade (see, e.g., \cite{PT, apt2004, APT UMZh 2005, Airey, Airey2} and references therein). To calculate the Hausdorff and packing dimension of sets defined in terms of Cantor series expansions it is extremely important to know whether the family of cylinders of the Cantor series expansion is faithful.

Let $\Phi_{k}$ be the family of the k-th rank closed intervals (cylinders) , i.e.,
$$\Phi_{k}:= \left\{ E:E=\Delta_{\alpha_{1}\alpha_{2}...\alpha_{k}}, ~\alpha_{i}\in \{0,...,n_{i}-1\}, ~i=1,~2,~...,~k \right\}$$
with
$$\Delta_{\alpha_{1}\alpha_{2}...\alpha_{k}}:=
\left\{x:x \in
\left[\sum \limits_{i=1}^{k}\frac{\alpha_{i}}{n_{1} n_{2} \ldots n_{i}}, \frac{1}{n_{1} n_{2} \ldots n_{k}}+\sum \limits_{i=1}^{k}\frac{\alpha_{i}}{n_{1} n_{2} \ldots n_{i}}
\right]\right\}.$$

Let $\Phi$ be the family of all possible closed intervals (cylinders), i.e.,

\[\Phi:=\left\{E:E=\Delta_{\alpha_{1}\alpha_{2}...\alpha_{k}}, ~k\in \mathbb{N}, ~\alpha_{i}\in\{0,...,n_{i}-1\}, ~i=1,~2,~...,~k\right\}.\]

In the paper \cite{ailt} authors found sharp conditions for the Hausdorff dimension faithfulness of the family $\Phi$.
The following theorem, being the main result of the paper,  gives necessary and sufficient condition for the packing dimension faithfulness of the family of cylinders generated by the Cantor series expansion. To the best of our knowledge this  is the first known sharp   condition of the packing dimension faithfulness for a class of packing families containing both faithful and non-faithful ones.

\begin{theorem}

The family $\Phi$ of Cantor coverings of the unit interval is faithful for the Packing dimension if and only if

\begin{equation} \label{eq:dost_umova_dov}
\mathop{\lim }\limits_{k\to \infty } \frac{\ln n_{k} }{\ln n_{1} \cdot n_{2} \cdot \ldots \cdot n_{k-1} } =0.
\end{equation}

\end{theorem}

\begin{proof} \textbf{Sufficiency.}  Let us show that  condition (\ref{eq:dost_umova_dov}) is sufficient for the faithfulness of $\Phi$ for the packing dimension calculation. Since the inequality $\dim_P(E,\Phi)\leqslant \dim_P(E)$ is true for an arbitrary covering family $\Phi$ and for a set $E \in [0,1]$, it is
sufficient to prove that
$$
\dim_P(E,\Phi) \geqslant \dim_P E.
$$

Let $\{E_j\}=(a_j;b_j)$ be an arbitrary centered $\varepsilon$-packing of a given set $E$ ($c_j:=\frac{a_j+b_j}{2} \in E$).
Then there exists a cylinder $\Delta_j:=\Delta(|E_j|) \in \Phi_{k_j}$ such that:

\begin{enumerate}
\item $\Delta_j \subset E_j$;
\item $c_j \in \Delta_j$;
\item $\forall \Delta \in \Phi_{k_{j}-1}: \Delta \not\subset E_j$.
\end{enumerate}


\noindent From the above it follows that $|\Delta_j| \geqslant \frac{1}{2n_{k_j}}\cdot |E_j|$.
Consequently, the coresponding $\alpha$-volume of $\varepsilon$-packing $\{\Delta_j\} \subset \Phi$ is bounded from below:
$$
\begin{aligned}
\sum_j|\Delta_j|^\alpha & \geqslant
\sum_j \left(\frac{1}{2n_{k_j}}\cdot |E_j|\right)^\alpha=
\sum_j \left(\frac{1}{2 n_{k_j}}\right)^\alpha \cdot |E_j|^{-\delta}\cdot |E_j|^{\alpha+\delta}\geqslant\\
& \sum_j \left(\frac{1}{2 n_{k_j}}\right)^\alpha \cdot |\Delta_j|^{-\delta}\cdot |E_j|^{\alpha+\delta}=
\sum_j \left(\frac{1}{2n_{k_j}}\right)^\alpha \cdot (n_1 n_2 \dots n_{k_j})^\delta \cdot |E_j|^{\alpha+\delta}
\end{aligned}
$$

Since $$\lim\limits_{i\to\infty}\frac{\ln n_i}{\ln(n_1 n_2 \dots n_{i-1})}=0,$$
 we have $\lim\limits_{i\to\infty}\frac{\ln n_i}{\ln(n_1 n_2 \dots n_i)}=0$, and, therefore,  $\lim\limits_{i\to\infty}\frac{\ln(n_1 n_2 \dots n_i)}{\ln n_i}=+\infty$.

 So,  for given $\alpha>0$, $\delta>0$, $\varepsilon>0$ there exists $m_0\in\mathbb{N}$ such that

\begin{enumerate}
\item $\frac{1}{n_1 n_2 \dots n_i}<\varepsilon, \forall i>m_0$;
\item $\frac{\delta \cdot \ln(n_1 n_2 \dots n_i)}{\alpha \cdot \ln n_i}>1, \forall i>m_0.$
\end{enumerate}

\noindent From what has already
been proved, it follows that
$$
\begin{aligned}
\sum_j|\Delta_j|^\alpha & \geqslant
\sum_j \frac{1}{2^\alpha} \cdot \left(e^{\frac{\delta \ln\left(n_1 n_2 \dots n_{k_j}\right)}{\alpha \ln(n_{k_j})}-1}\right)^{\alpha \ln(n_{k_{j}})}
\cdot |E_j|^{\alpha+\delta}\geqslant \frac{1}{2^\alpha} \sum_j |E_j|^{\alpha+\delta},
\end{aligned}
$$
\noindent Consequently,
$$
\mathcal{P}^\alpha_\varepsilon(E,\Phi) \geqslant \frac{1}{2^\alpha} \sum_j |E_j|^{\alpha+\delta}, \forall \alpha>0, \forall \delta>0,
\forall \varepsilon>0
$$
\noindent and
\begin{equation}
\label{eq:cantor_pack_mes_ineq}
\mathcal{P}^\alpha(E)\geqslant \mathcal{P}^\alpha(E,\Phi) \geqslant \frac{1}{2^\alpha} \mathcal{P}^{\alpha+\delta}(E).
\end{equation}
Let $\alpha_0:=\dim_P(E,\Phi)$. Then
$$
\mathcal{P}^{\alpha+\frac1n}(E,\Phi)=0, \forall n\in\mathbb{N}.
$$
Since $\mathcal{P}^{\alpha+\delta}(E) \leqslant 2^\alpha \mathcal{P}^\alpha(E,\Phi)$,
we get
$$
\mathcal{P}^{\alpha+\delta+\frac1n}(E)=0, \forall n\in\mathbb{N}, \forall \delta>0.
$$
So,
$$
\dim_P E \leqslant \alpha_0+\delta+\frac1n, \forall n\in\mathbb{N}, \forall \delta>0
$$
and $\dim_P E \leqslant \dim_P(E,\Phi)$.

\textbf{Necessity.}
Let us show that  condition (\ref{eq:dost_umova_dov}) is necessary for the faithfulness of $\Phi$ for the packing dimension calculation.

Suppose, contrary to our claim,  that $\Phi$ is faithful and  condition (\ref{eq:dost_umova_dov}) does not hold. Then
\begin{equation} \label{eq:upper_limit_c}
\mathop{\overline{\lim }}\limits_{k\to \infty } \frac{\ln n_{k} }{\ln n_{1} \cdot n_{2} \cdot \ldots \cdot n_{k-1} } =:C>0.
\end{equation}


From (\ref{eq:upper_limit_c}) it follows that there exists an increasing subsequence $\{k_s\}$ such that
\begin{equation}
\label{eq:first_condition_for_ks}
\lim_{s\to\infty}\frac{\ln n_{k_s}}{\ln(n_1 n_2 \dots n_{k_s-1})}=C,
\end{equation}

and

\begin{equation}
\label{eq:second_condition_for_ks}
\lim_{s\to\infty}\frac{\ln(n_{k_1} n_{k_2} \dots n_{k_{s-1}})}{\ln(n_1 n_2 \dots n_{k_s-1})}=0.
\end{equation}

Let us construct a set $T^*$ such that
\begin{equation}
\label{eq:ineq_for_t_star}
\dim_P(T^*,\Phi)<\dim_P(T^*).
\end{equation}
Let $A=\{k_1,k_2,\dots,k_s,\dots\}$ and $$T^*=\left\{x: x=\Delta_{\alpha_1 \alpha_2 \dots \alpha_k\dots}\right\},$$ where $\alpha_j=0$ if $j\notin A$,

and $\alpha_j \in \left\{0,[\sqrt{n_j}], 2[\sqrt{n_j}], \dots, ([\sqrt{n_j}]-1)\cdot[\sqrt{n_j}]\right\}$ if $j\in A$.

Firstly let us show that $\dim_P(T^*)\geqslant \frac{C}{C+2}$.

To this aim  for a given  $\varepsilon>0$ let us choose $m_0$ such that $\frac{1}{n_1 n_2 \dots n_{m_0}}<\varepsilon$.

It is clear that for an arbitrary $s>m_{0}$ the set $T^*$ can be packed by
$$
\left[\sqrt{n_{k_1}}\right]\cdot
\left[\sqrt{n_{k_2}}\right]\cdot
\dots \cdot
\left[\sqrt{n_{k_s}}\right]=:Q_s
$$
intervals and each of them is a union of $[\sqrt{n_k}]$ cylinders from $\Phi_{k_{s}}$.

A length of each interval equals
$$
\frac{\left[\sqrt{n_{k_s}}\right]}{n_1 n_2 \dots n_{k_s}}=:V_s.
$$
The $\alpha$-volume of this $\varepsilon$-packing is equal to
$$
Q_s\cdot (V_s)^\alpha=\\ \left(
\exp\left(
\frac{\ln Q_s}{\ln\left(n_1 n_2 \dots n_{k_s-1}\right)}+
\frac{\alpha\ln\left[\sqrt{n_{k_s}}\right]}{\ln\left(n_1 n_2 \dots n_{k_s-1}\right)}-
\frac{\alpha\ln\left(n_1 n_2 \dots n_{k_s}\right)}{\ln\left(n_1 n_2 \dots n_{k_s-1}\right)}
\right)
\right)^{\ln\left(n_1 n_2 \dots n_{k_s-1}\right)} .
$$
Let us calculate
$$
\lim_{s\to\infty} \frac{\ln Q_s}{\ln\left(n_1 n_2 \dots n_{k_s-1}\right)}=
\lim_{s\to\infty} \frac
{\ln\left[\sqrt{n_{k_1}}\right]+\ln\left[\sqrt{n_{k_2}}\right]+\dots\ln\left[\sqrt{n_{k_{s-1}}}\right]+\ln\left[\sqrt{n_k}\right]}
{\ln\left(n_1 n_2 \dots n_{k_s-1}\right)}=\frac{C}{2}.
$$
Consequently,
$$
\lim_{s\to\infty} Q_s \cdot (V_s)^\alpha=\left(exp(C-2\alpha-\alpha C)\right)^{\frac{1}{2}\ln\left(n_1 n_2 \dots n_{k_s-1}\right)}
$$
If $C-2\alpha -\alpha C>0$ ($\alpha>\frac{C}{2+C}$), then $Q_s (V_s)^\alpha\to+\infty$ as $s\to\infty$. Therefore, if $\alpha<\frac{C}{2+C}$, then  $\mathcal{P}^\alpha_\varepsilon(T^*)=+\infty$, $\forall \varepsilon>0$. Since $T^*$ is a nowhere dense closed set, we have  $\mathcal{P}^\alpha(T^*)=+\infty$, $\forall \alpha<\frac{C}{2+C}$. Therefore, we have
\begin{equation}
\label{eq:dimp_of_T_star}
\dim_P T^*\geqslant \frac{C}{2+C}.
\end{equation}

On the other hand we shall prove that
$$
\dim_P(T^*,\Phi)\leqslant \frac{C}{2C+2}.
$$
Let $\mu_\xi$ be the probability measure corresponding to the random variable $\xi$ with independent digits of the Cantor series expansion, i.e.,
$$
\xi =\sum_{k=1}^{\infty }\frac{\xi_{k}}{n_{1} n_{2} \ldots n_{k}},
$$

where $\xi _{k} $ are independent random variables such that

$\xi_k$ takes the value $0$ with probability 1,  if $k\notin A$;

if $k\in A$, then

\bigskip

\begin{tabular}{|c|c|c|c|c|c|}
\hline
$\xi_k$ &
0 &
$\left[\sqrt{n_k}\right]$ &
$2\cdot\left[\sqrt{n_k}\right]$ &
\dots &
$(\left[\sqrt{n_k}\right]-1)\cdot \left[\sqrt{n_k}\right]$ \\
\hline
 &
$\frac{1}{\left[\sqrt{n_k}\right]}$ &
$\frac{1}{\left[\sqrt{n_k}\right]}$ &
$\frac{1}{\left[\sqrt{n_k}\right]}$ &
\dots &
$\frac{1}{\left[\sqrt{n_k}\right]}$ \\
\hline
\end{tabular}.

\bigskip
Let $\lambda$ be Lebesgue measure on $[0;1]$ and $\Delta_k(x)$ be cylinder of $k$-th  rank, where $x \in \Delta_k(x)$. Then $\forall x\in T^*$:
$$
\begin{aligned}
& \ln \mu_\xi (\Delta_{k_s}(x))=-\ln Q_s; \\
& \ln \lambda(\Delta_{k_s}(x))=-\ln (n_1 n_2 \dots n_{k_s}); \\
& \lim_{s\to\infty} \frac{\ln \mu_\xi (\Delta_{k_s}(x))}{\ln \lambda(\Delta_{k_s}(x))}
=\lim_{s\to\infty} \frac
{\ln\left[\sqrt{n_{k_1}}\right]+\ln\left[\sqrt{n_{k_2}}\right]+\dots\ln\left[\sqrt{n_{k_{s-1}}}\right]+\ln\left[\sqrt{n_k}\right]}
{\ln (n_1 n_2 \dots n_{k_{s-1}})+\ln n_{k_s}}=\frac{C}{2C+2}.
\end{aligned}
$$

Let us prove that
$$
\limsup_{l\to\infty} \frac{\ln \mu_\xi (\Delta_l(x))}{\ln \lambda(\Delta_l(x))}=\frac{C}{2C+2}, \forall x\in T^*.
$$
For an arbitrary $l\in \mathbb{N}$ there exists  a number $s=s(l)$ such that $k_s\leqslant l<k_{s+1}$. Fix $d_l:=l-k_s$. Then $l=k_s+d_l$, where $d_l<k_{s+1}-k_s$. Then
$$
\begin{aligned}
\frac{\ln \mu_\xi (\Delta_l(x))}{\ln \lambda(\Delta_l(x))}&=
\frac
{\ln \left(\mu_\xi (\Delta_{k_s}(x))\cdot \overbrace{1\cdot 1 \cdot \dots \cdot 1}^{d_l}\right)}
{\ln \left(\lambda (\Delta_{k_s}(x))\cdot \frac{1}{n_{k_s+1}}\cdot \frac{1}{n_{k_s+2}}\cdot \dots \cdot\frac{1}{n_{k_s+d_l}}\right)}
=\\
&=\frac{\ln\left( \frac{1}{\mu_\xi (\Delta_{k_s}(x))}\right)+0}
{\ln\left(\frac{1}{\lambda(\Delta_{k_s}(x))}\right)+\sum_{i=1}^{d_l} \ln n_{k_s+i}}\leqslant\\
&\leqslant\frac{\ln\left( \frac{1}{\mu_\xi (\Delta_{k_s}(x))}\right)}{\ln\left(\frac{1}{\lambda(\Delta_{k_s}(x))}\right)}=
\frac{\ln \mu_\xi (\Delta_{k_s}(x))}{\ln \lambda(\Delta_{k_s}(x))}.
\end{aligned}
$$
Consequently, $\gamma_i:=\frac{\ln \mu_\xi (\Delta_{k_s+i}(x))}{\ln \lambda(\Delta_{k_s+i}(x))}$ is a monotonically decreasing for $i\in\{0,1,2,\dots k_{s+1}-k_s-1\}$. So,
$$
\limsup_{l\to\infty} \frac{\ln \mu_\xi (\Delta_l(x))}{\ln \lambda(\Delta_l(x))}=
\limsup_{s\to\infty} \frac{\ln \mu_\xi (\Delta_{k_s}(x))}{\ln \lambda(\Delta_{k_s}(x))}=\frac{C}{2C+2}.
$$
This shows that
$$
T^*\subset \left\{
x: \limsup_{l\to\infty} \frac{\ln \mu_\xi (\Delta_l(x))}{\ln \lambda(\Delta_l(x))}\leqslant \frac{C}{2C+2}
\right\}.
$$
By Theorem \ref{billingsley_theorem_for_dimp}, we get
$$
\dim_P(T^*,\Phi,\lambda)\leqslant \frac{C}{2C+2}\cdot \dim_P(T^*,\Phi,\mu_\xi).
$$
Since $T^*$ is the topological support of the measure $\mu_\xi$, we conclude that $\dim_P(T^*,\Phi,\mu)=1$. Consequently, $\dim_P(T^*,\Phi,\lambda)=\dim_P(T^*,\Phi)$. This shows that
\begin{equation}
\label{eq:dimp_Phi_of_T_star}
\dim_P(T^*,\Phi)\leqslant \frac{C}{2C+2}.
\end{equation}

From inequalities \eqref{eq:dimp_of_T_star} and \eqref{eq:dimp_Phi_of_T_star}, it follows that
$$
\dim_P(T^*,\Phi) \leqslant \frac{C}{2C+2} < \frac{C}{C+2} \leqslant \dim_P T^*,
$$
which completes the proof.
\end{proof}

\bigskip
\textbf{Acknowledgment.}
This work was partly supported by  SFB-701 <<Spectral Structures and Topological Methods in Mathematics>> (Bielefeld University), STREVCOM FP-7-IRSES 612669 project and by the Alexander von Humboldt Foundation.

\end{document}